\documentclass[a4paper,reqno]{amsart}

\usepackage{amssymb, longtable}
\usepackage[mathcal]{euscript}

\usepackage{color}

\usepackage{hyperref}                  

\newcommand{\bydef}{:=}

\DeclareMathOperator{\espan}{\mathrm{span}}

\newcommand{\cA}{\mathcal{A}}

\newcommand{\cC}{\mathcal{C}}

\newcommand{\cK}{\mathcal{K}}
\newcommand{\cL}{\mathcal{L}}

\newcommand{\cQ}{\mathcal{Q}}

\newcommand{\cS}{\mathcal{S}}

\newcommand{\cU}{\mathcal{U}}
\newcommand{\cV}{\mathcal{V}}

\newcommand{\frg}{{\mathfrak g}}

\newcommand{\ZZ}{\mathbb{Z}}

\newcommand{\RR}{\mathbb{R}}

\newcommand{\OO}{\mathbb{O}}
\newcommand{\FF}{\mathbb{F}}

\DeclareMathOperator{\End}{\mathrm{End}}

\DeclareMathOperator{\Aut}{\mathrm{Aut}}

\DeclareMathOperator{\Der}{\mathrm{Der}}
\DeclareMathOperator{\LocDer}{\mathrm{LocDer}}
\DeclareMathOperator{\twoLocDer}{\mathrm{2LocDer}}

\newcommand{\frso}{{\mathfrak{so}}}
\newcommand{\frpsl}{{\mathfrak{psl}}}

\numberwithin{equation}{section}

\newcommand{\rojo}{} 

\newcommand{\subo}{_{\bar 0}}

\newcommand{\nup}{\textup{n}}

\newtheorem{theorem}{Theorem}[section]

\newtheorem*{proposition*}{Proposition}
\newtheorem{lemma}[theorem]{Lemma}
\newtheorem{corollary}[theorem]{Corollary}

\theoremstyle{definition}

\theoremstyle{remark} \newtheorem{remark}[theorem]{Remark}

\begin{document}

\title{Local and $2$-local derivations of Cayley algebras}

\author[Sh. A. Ayupov]{Shavkat Ayupov$^{1,2}$}
\address{$^{1}$V.I.Romanovskiy Institute of Mathematics\\
	Uzbekistan Academy of Sciences\\ University street, 9, Olmazor district, Tashkent, 100174, Uzbekistan}
\address{$^2$National University of Uzbekistan \\
	4, Olmazor district, Tashkent, 100174, Uzbekistan}
\email{\textcolor[rgb]{0.00,0.00,0.84}{shavkat.ayupov@mathinst.uz}}

\author[A.Elduque]{Alberto Elduque$^3$}
\address{$^{3}$Departamento de
Matem\'{a}ticas e Instituto Universitario de Matem\'aticas y
Aplicaciones, Universidad de Zaragoza, 50009 Zaragoza, Spain}
\email{\textcolor[rgb]{0.00,0.00,0.84}{elduque@unizar.es}}
\thanks{${}^3$%
Supported by grants MTM2017-83506-C2-1-P (AEI/FEDER, UE) and E22\_17R
(Gobierno de Arag\'on, Grupo de referencia ``\'Algebra y
Geometr{\'\i}a'', cofunded by Feder 2014-2020 ``Construyendo Europa desde Arag\'on'')}

\author[K. Kudaybergenov]{Karimbergen Kudaybergenov$^{1,4}$}
\address{$^4$Department of Mathematics\\
	Karakalpak State University\\
	1, Ch. Abdirov,   Nukus, 230112, Uzbekistan}
\email{\textcolor[rgb]{0.00,0.00,0.84}{karim2006@mail.ru}}




\date{\today}

\begin{abstract}
The present paper is devoted to the description  of local and $2$-local derivations
on Cayley algebras over an arbitrary field $\FF$.
Given a Cayley  algebra $\cC$
with  norm $\nup$, let $\cC_0$ be its subspace of trace $0$ elements.
We prove that the space of all local derivations of $\cC$ coincides with the Lie algebra
$\{d\in\frso(\cC,\nup)\mid d(1)=0\}$ which is isomorphic to the  orthogonal Lie algebra
$\frso(\cC_0,\nup)$.
Further we prove that, {\rojo surprisingly, the behavior of $2$-local
derivations depends on the Cayley algebra being split or division. Every}
$2$-local derivation on the split Cayley algebra is a derivation, i.e. they form the exceptional Lie algebra $\frg_2(\FF)$ if
$\textrm{char}\FF\neq 2,3$.
 On the other hand, on  division Cayley algebras over a  field $\FF$,  the sets of $2$-local derivations and local derivations coincide, and they are isomorphic to  the Lie algebra $\frso(\cC_0,\nup)$.
As a corollary we obtain  descriptions of local and $2$-local derivations of {\rojo the seven dimensional simple non-Lie Malcev algebras} over  fields of {\rojo characteristic} $\neq 2,3$.
\end{abstract}


\maketitle

\section{Introduction}

Let $\cA$ be an   algebra (not necessary associative). Recall that a linear mapping $d: \cA\to \cA$ is said to be a derivation, if $d(xy)=d(x)y+xd(y)$ for all $x, y\in \cA$.
A  linear mapping  $\Delta$ is said to be a
local  derivation, if for every $x\in \cA$ there exists a derivation $d_x$ on $\cA$ (depending on $x$) such that $\Delta(x)=d_x(x)$.
This notion was  introduced and investigated independently by R.V. Kadison~\cite{Kadison90} and D.R.
Larson and A.R. Sourour~\cite{Larson90}.
The above papers gave rise to a series of works devoted to the description of mappings which are close to automorphisms and derivations of $C^\ast$-algebras and  operator algebras.

In 1997, P.\v{S}emrl \cite{Semrl97} introduced the concepts of
$2$-local derivations and $2$-automorphisms.
Recall that a mapping $\Delta:\cA \to \cA$ (not necessary linear) is said to be a $2$-local derivation, if for every pair $x,y\in \cA$ there exists a derivation $d_{x,y}:\cA \to \cA$ (depending on $x,y$) such that
$\Delta(x)=d_{x,y}(x)$, $\Delta(y)=d_{x,y}(y)$.

P. \v{S}emrl  \cite{Semrl97} {\rojo described} $2$-local derivations  on the algebra $B(H)$  of all bounded linear operators on the infinite-dimensional separable Hilbert space $H$,  by proving that every $2$-local  derivation on $B(H)$  is a  derivation.
A detailed discussion of $2$-local derivations on operator algebras can be found in the survey \cite{AKR}.

In \cite{AKA21}  a general form of local derivations on the real Cayley   algebra $\OO$ was obtained.
This description  implies that the space of all local derivations on $\OO$,  when equipped with Lie bracket, is isomorphic
to the Lie algebra $\frso_7(\RR)$ of all real skew-symmetric $7\times 7$-matrices.
Also it is proved that   any $2$-local derivation on a Cayley  algebra
$\cC$ over an algebraically closed field $\FF$  is a derivation.

The present paper is devoted to the description of local and $2$-local derivations {\rojo on Cayley algebras} over an arbitrary field, and it is organized as follows.

In Section 2 we give the necessary information concerning Cayley algebras and their derivations.

In Section 3 we find  a general form for the local derivations on a Cayley  algebra $\cC$ over an arbitrary field $\FF$  with a norm $\nup$. Namely, we prove that  the space of local derivations of $\cC$ is the Lie algebra
$\{d\in\frso(\cC,\nup)\mid d(1)=0\}$ (Theorem~\ref{th:main}).

In Section 4 we consider $2$-local derivations on {\rojo a} Cayley  algebra $\cC$.
{\rojo It turns out that the structure} of $2$-local derivations on a Cayley algebra $\cC$ depends on the norm $\nup$. If the norm $\nup$ is isotropic, then $\cC$ is the split Cayley algebra and  all $2$-local derivations on $\cC$ are derivations (Theorem~\ref{2localsplit}).
Further, if the norm $\nup$ is anisotropic, then $\cC$ is a division  Cayley algebra, and
n this case the spaces of all $2$-local
derivations and local derivations  on $\cC$ coincide, and both are isomorphic to  $\frso(\cC_0,\nup)$ (Theorem~\ref{th:2local_division}).

Section 5 is devoted to applications of the above results to the  description of local and $2$-local derivations of the seven dimensional simple Malcev algebras
over a field of {\rojo characteristic $\neq 2$.}

\section{Cayley algebras}

Let $\FF$  be an arbitrary field. Cayley {\rojo (or octonion)} algebras  over $\FF$
constitute a well-known class
of nonassociative algebras. They are unital nonassociative
algebras $\cC$  of dimension eight over $\FF$,
endowed with a nonsingular quadratic multiplicative form (the norm)
$\nup : \cC \to  \FF$. Hence
$$
\nup(xy)=\nup(x)\nup(y)
$$
for all  $x, y \in  \mathcal{C}$, and the polar form
$$
\nup(x, y)\bydef\nup(x+y)-\nup(x)-\nup(y)
$$
is a nondegenerate bilinear form.

Any element in a Cayley algebra $\cC$ satisfies the degree $2$ equation:
\begin{align}\label{xsqn}
	x^2-\nup(x,1)x+\nup(x)1=0.
\end{align}
The map $x \to  \overline{x}=\nup(x, 1)1-x$ is an involution
and the
trace $t(x)=\nup(x, 1)$ and norm $\nup(x)$ are given by $t(x)1=x+\overline{x}$,
$\nup(x)1 = x\overline{x}=\overline{x}x$ for all $x \in \cC$.

Note that two Cayley algebras $\mathcal{C}_1$ and $\mathcal{C}_2$, with respective norms $\nup_1$
and $\nup_2$, are isomorphic if and only if the norms {\rojo $\nup_1$ and $\nup_2$} are
isometric (see \cite[Corollary 4.7]{EKmon}). {\rojo It must be remarked that Cayley
algebras are alternative, that is, the subalgebra generated by any two elements
is associative.}

From \eqref{xsqn}, we get
\begin{align}\label{jorprod}
	xy+yx-\nup(x,1)y-\nup(y,1)x +\nup(x,y)1=0,
\end{align}
for all $x,y\in\cC$.

Recall that the  norm $\nup$  is isotropic if there is a non zero element $x\in \cC$ with $\nup(x)=0$, otherwise its called anisotropic.
Note that any Cayley algebra with anisotropic norm is a division algebra.

It is known that, up to isomorphism, there is a unique Cayley algebra whose
norm is isotropic. It is called the split Cayley algebra. A split Cayley
$\cC$ admits a \emph{canonical basis} $\{e_1, e_2, u_1, u_2, u_3, v_1, v_2, v_3\}$. {\rojo The multiplication table in this basis is given in Table
\ref{table:1} (see \cite[\S 4.1]{EKmon}).}
\begin{table}[h!]
\caption{}
\label{table:1}
	\centering
	\begin{tabular}{ | c | c c | c c c | c c c |}
		\hline
		& $e_1$ & $e_2$ & $u_1$ & $u_2$ & $u_3$ & $v_1$ & $v_2$ & $v_3$ \\
		\hline
		$e_1$ & $e_1$ & 0  & $u_1$ & $u_2$ & $u_3$ & 0 & 0 & 0 \\
		$e_2$ & 0 & $e_2$  & 0 & 0 &  0 &  $v_1$ & $v_2$ & $v_3$ \\
		\hline
		$u_1$ & 0  & $u_1$ & 0 & $v_3$ & $-v_2$  & $-e_1$ & 0 & 0 \\
		$u_2$ & 0 & $u_2$  & $-v_3$ & 0 & $v_1$ & 0 & $-e_1$ & 0 \\
		$u_3$ & 0 & $u_3$  & $v_2$ & $-v_1$ & 0 & 0 & 0 & $-e_1$ \\
		\hline
		$v_1$ & $v_1$ & 0  & $-e_2$ & 0 & 0 & 0 & $u_3$  & $-u_2$ \\
		$v_2$ & $v_2$ & 0  & 0 & $-e_2$ & 0 & $-u_3$ & 0 & $u_1$ \\
		$v_3$ & $v_3$ & 0  & 0 & 0     & $-e_2$ & $u_2$ & $-u_1$ & 0  \\
		\hline
	\end{tabular}
\end{table}

Recall too that $\cK=\FF e_1+ \FF e_2$, which is isomorphic to $\FF\times \FF$, is the split Hurwitz
algebra of dimension $2$, and with $\cU = \FF u_1+\FF u_2+\FF u_3$ and $\cV = \FF v_1+\FF v_2+\FF v_3$,
the decomposition $\cC=\cK\oplus \cU\oplus \cV$ is a $\ZZ_3$-grading:
$\cC_{\overline{0}}=\cK, \, \cC_{\overline{1}}=\cU,\, \cC_{\overline{2}}=\cV$.
This induces a $\ZZ_3$-grading on the Lie algebra of derivations
$\frg=\Der(\cC)$, with
$\frg_i = \left\{d \in \frg \mid d(\cC_i) \subset  \cC_{i+j}\,\,\, \textrm{for all}\,\,\, j \in  \ZZ_3\right\}$.

Let $\cC$ be a Cayley (or octonion) algebra over an arbitrary field $\FF$.
Note that any derivation on $\mathcal{C}$ is skew-symmetric with respect to the norm $\nup$, i.e., {\rojo $\nup\bigl(x,d(x)\bigr)=0$ for all $x\in\cC$, and this implies}
\begin{align}\label{der-skew}
\nup(d(x), y)+\nup(x, d(y))=0
\end{align}
for all $x, y\in \cC$.

{\rojo It is well-known that, over a field $\FF$ of characteristic $\neq 2,3$, the Lie algebra of all derivations $\Der(\cC)$ of a Cayley algebra
$\cC$
is a central simple Lie algebra of type $G_2$ and, conversely, any central simple Lie algebra of type $G_2$ is obtained, up to isomorphism, in this way. (See
\cite[Theorem IV.4.1]{Seligman} and the references there in.)

Over fields of characteristic $3$, $\Der(\cC)$ is no longer simple, but has a unique
minimal ideal, which is simple and isomorphic to the Lie algebra $\cC_0$ of trace
zero elements with the Lie bracket $[x,y]=xy-yx$, which is a central
simple Lie algebra of type $A_2$. The quotient modulo this unique
minimal ideal is again isomorphic to $\cC_0$. On the other hand, over fields of
characteristic $2$, $\Der(\cC)$ is isomorphic to the split simple Lie algebra
$\frpsl_4(\FF)$ of type $A_3$, and surprisingly this does not depend on whether
$\cC$ is the split algebra or a division algebra. (See \cite[\S\S 3,4]{C-RE16} and the references there in.)}

\smallskip

The arguments in \cite[\S 4.4]{EKmon} show that the coordinate matrix in this basis of any derivation of $\cC$ is of the form:
\begin{equation}\label{eq:derivation_matrix}
	\begin{pmatrix}
		0&0&\alpha'&\beta'&\gamma'&-\alpha&-\beta&-\gamma\\
		0&0&-\alpha'&-\beta'&-\gamma'&\alpha&\beta&\gamma\\
		\alpha&-\alpha&a_{11}&a_{12}&a_{13}&0&\gamma'&-\beta'\\
		\beta&-\beta&a_{21}&a_{22}&a_{23}&-\gamma'&0&\alpha'\\
		\gamma&-\gamma&a_{31}&a_{32}&a_{33}&\beta'&-\alpha'&0\\
		-\alpha'&\alpha'&0&\gamma&-\beta&-a_{11}&-a_{21}&-a_{31}\\
		-\beta'&\beta'&-\gamma&0&\alpha&-a_{12}&-a_{22}&-a_{32}\\
		-\gamma'&\gamma'&\beta&-\alpha&0&-a_{13}&-a_{23}&-a_{33}
	\end{pmatrix},
\end{equation}
with $a_{11}+a_{22}+a_{33}=0$.

\

\section{Local derivations on Cayley  algebras}


The main result of this section is the following theorem.

\begin{theorem}\label{th:main}
	Let $\cC$ be a {\rojo Cayley  algebra} over an arbitrary field with norm
	$\nup$.
	Then the space of local derivations of $\cC$ is the Lie algebra
	$\{d\in\frso(\cC,\nup)\mid d(1)=0\}$.
\end{theorem}

\begin{remark}
	Let $\cC_0$ be the subspace of trace $0$ elements.
	If the characteristic of $\FF$ is not $2$, then $\cC=\FF 1\oplus\cC_0$.
	Embed $\End_\FF(\cC_0)$ inside $\End_\FF(\cC)$ by extending any endomorphism $\varphi$ of $\cC_0$ to an endomorphism of $\cC$ by imposing $\varphi(1)=0$. Then
	(characteristic not $2$) the Lie algebra in Theorem \ref{th:main} is naturally identified with $\frso(\cC_0,\nup)=\{d\in \End_\FF(\cC_0)\mid
	\nup\bigl(v,d(v)\bigr)=0\ \forall v\in \cC_0\}$. The trace of any element
	in $\frso(\cC_0,\nup)$ is zero.
	
	However, if the characteristic of $\FF$ is $2$, then $1\in\cC_0$. Here
	the radical of
	the restriction of the polarization of $\nup$ to $\cC_0$ is $\FF 1$. In this case the right definition of $\frso(\cC_0,\nup)$ is
	$\{d\in \End_\FF(\cC_0)\mid \operatorname{trace}(d)=0\ \text{and}\
	\nup\bigl(v,d(v)\bigr)=0\ \forall v\in \cC_0\}$. It follows that any
	element in $\frso(\cC_0,\nup)$ annihilates $1$. Moreover, any element
	$d\in\frso(\cC,\nup)$ with $d(1)=0$ leaves $\cC_0$ invariant, and its restriction to $\cC_0$ lies in $\frso(\cC_0,\nup)$. Conversely, pick any element
	$u\in\cC\setminus\FF 1$, so that $\cC=\FF u\oplus\cC_0$.  Any element
	$d\in\frso(\cC_0,\nup)$ can be extended to an element of
	$\frso(\cC,\nup)$ by defining $d(u)$ as the unique element of $\cC$
	such that $\nup\bigl(d(u),u\bigr)=0$ and
	$\nup\bigl(d(u),x\bigr)=\nup\bigl(u,d(x)\bigr)$ for all $x\in \cC_0$. Therefore, the restriction
	\[
	\begin{split}
	\{d\in\frso(\cC,\nup)\mid d(1)=0\}&\longrightarrow \frso(\cC_0,\nup)\\
	d\quad&\mapsto\quad d\vert_{\cC_0}
	\end{split}
	\]
	is an isomorphism of Lie algebras also in characteristic $2$.
	
	Therefore, in any characteristic, we can identify canonically the Lie algebra
	$\{d\in\frso(\cC,\nup)\mid d(1)=0\}$ with $\frso(\cC_0,\nup)$.
\end{remark}

\smallskip

The proof of Theorem \ref{th:main} is based on two lemmas. The first one
is surely well-known. {\rojo For Cayley division algebras over fields of characteristic not two, it follows from the Cayley-Dickson doubling process as in
\cite[Corollary 1.7.2]{SV2000}. We include a proof, valid over arbitrary fields, for completeness.}

\begin{lemma}\label{le:orbits}
	Let $\cC$ be a Cayley  over a field $\FF$ with norm $\nup$. Any two elements of
	$\cC\setminus \FF 1$ are conjugate under $\Aut\cC$ if and only if they have the same norm and the same trace.
\end{lemma}

\begin{proof}
Let $a,b\in\cC\setminus\FF 1$ with $t(a)=t(b)=\alpha$, and $\nup(a)=\nup(b)=\mu$.
The minimal polynomial of both $a$ and $b$ is $p(X)=X^2-\alpha X+\mu$.
	
If $p(X)$ is separable, then $\cK_1=\FF 1+\FF a$ and $\cK_2=\FF 1 +\FF b$ are isomorphic
	two-dimensional composition subalgebras of $\cC$, and we use the Cayley-Dickson
	doubling process to extend any isomorphim between $\cK_1$ and $\cK_2$ carrying $a$ to $b$ to an automorphism of $\cC$.
	
	Otherwise, either the characteristic of $\FF$ is not $2$ and
	$p(X)=\bigl(X-\frac{\alpha}{2}\bigr)^2$, or the characteristic of $\FF$ is $2$ and
	$\alpha=0$.
	
	Note that if $x\in\cC\setminus\FF 1$ satisfies $x^2=0$ and we pick $y\in\cC$ with
	$\nup(x,\bar y)=1$, then
	$1=\nup(x,\bar y)=-\bar x y-\bar y x=xy-\bar y x,
	$ and then
	$e_1=xy$ is a nonzero idempotent, because
	$e_1^2=xyxy=(1+\bar y x)xy=xy+\bar yx^2y=e_1$. Also, we have
	$e_1x=xyx=(1+\bar yx)x=x$ and $xe_1=x^2y=0$. Hence $x$ lies in the Peirce component
	$\{u\in\cC\mid e_1x=x,\ xe_1=0\}=e_1\cC e_2$, and as in \cite[\S 4.4]{EKmon} we may take a
	canonical basis of $\cC$ with $x=u_1$. In particular, if $\alpha=0=\mu$, then $a$ and $b$
	are both conjugate to $u_1$.
	
	Also, if the characteristic of $\FF$ is not $2$ and
	$p(X)=\bigl(X-\frac{\alpha}{2}\bigr)^2$, then there are automorphisms $\varphi,\psi$ of $\cC$ such that $\varphi(a-\frac{\alpha}{2}1)=u_1=\psi(b-\frac{\alpha}{2}1)$,
	and hence $\psi^{-1}\circ\varphi(a)=b$.
	
	Finally, if the characteristic of $\FF$ is $2$ and $\alpha=0$ but $\mu\neq 0$, we distinguish two cases:
	\begin{itemize}
		\item If $1\in\FF a+\FF b$, then $b=\epsilon 1+\delta a$ with $\delta\neq 0$, so
		$\mu=\nup(a)=\nup(b)=\epsilon^2+\delta^2\nup(a)$. Hence
		$(1+\delta)^2\nup(a)=\epsilon^2$
		and either $\delta=1$, $\epsilon=0$ and $a=b$, or $\delta\neq 1$ and
		$\mu=\nup(a)=\nup(b)\in\FF^2$. Write $\mu=\gamma^2$. In this case,
		$(a+\gamma 1)^2=0=(b+\gamma 1)^2$ and again there are automorphisms
		$\varphi,\psi$ of $\cC$ such that $\varphi(a+\gamma 1)=u_1=\psi(b+\gamma 1)$,
		so that $\psi^{-1}\circ\varphi(a)=b$.
		
		\item If, on the contrary, $1\not\in\FF a+\FF b$, then $(\FF a+\FF b)^\perp
		\not\subseteq (\FF 1)^\perp$, so we may pick  an element $u\in\cC$ such that
		$\nup(u,1)=1$ and $\nup(u,a)=0=\nup(u,b)$. Then $\cK=\FF 1+\FF u$ is a
		two-dimensional composition subalgebra of $\cC$ and, by the Cayley-Dickson doubling
		process, the quaternion subalgebra $\cQ_1=\cK\oplus\cK a$ and $\cQ_2=\cK\oplus\cK b$
		are isomorphic, under an isomorphism that is the identity on $\cK$ and takes $a$ to
		$b$. Again, by the Cayley-Dickson doubling process, this isomorphism can be
		extended to an automorphism of $\cC$. (See, for instance,
		\cite[Corollary 1.7.2]{SV2000}.) \qedhere
	\end{itemize}
\end{proof}

\begin{lemma}\label{le:derC}
	Let $\cC$ be a Cayley algebra over an arbitrary field $\FF$, with Lie algebra of derivations $\Der(\cC)$, and let $x$ be any element in $\cC\setminus\FF 1$. Then
	$$
	\Der(\cC)x\bydef\{d(x)\mid d\in\Der(\cC)\}
	=(\FF 1+\FF x)^\perp=\bigl(\FF x\bigr)^\perp\cap \cC_0,
	$$
where $\perp$ denotes the orthogonal subspace relative to the norm $\nup$.
\end{lemma}

\begin{proof}
	We may extend scalars and assume that $\FF$ is algebraically closed. In particular, $\cC$ is split. Take a canonical basis of $\cC$ as in Table~\ref{table:1}. If the minimal polynomial of $x$ is separable:
	$p(X)=(X-\alpha)(X-\beta)$, with $\alpha\neq\beta$, then Lemma \ref{le:orbits}
	shows that $x$ is conjugate to $\alpha e_1+\beta e_2$. But
Equation \eqref{eq:derivation_matrix} shows that
\[
	\Der(\cC)(\alpha e_1+\beta e_2)=\cU\oplus
	\cV=(\FF e_1+\FF e_2)^\perp
	=\bigl(\FF 1+\FF (\alpha e_1+\beta e_2)\bigr)^\perp
\]
and we are done in this case.
	
	Otherwise the minimal polynomial of $x$ is of the form
$p(X)=(X-\lambda)^2$ for some $\lambda\in \FF$. By Lemma \ref{le:orbits}, $x-\lambda 1$ is conjugate to $u_1$, and hence $x$ is conjugate to $\lambda 1+u_1$. Now Equation \eqref{eq:derivation_matrix} shows that
\[
	\Der(\cC)(\lambda 1+u_1)=\Der(\cC)u_1
	=\FF (e_1-e_2)\oplus \cU\oplus \FF v_2\oplus\FF v_3=(\FF 1+\FF u_1)^\perp,
\]
	and we are done.
\end{proof}

\bigskip

\begin{proof}[Proof of Theorem \ref{th:main}.] If $\Delta$ is a local derivation of $\cC$,
$\Delta(1)=0$ and for any $x\in\cC$, we get $\nup\bigl(x,\Delta(x)\bigr)=0$, as
$\Delta(x)=d(x)$ for a derivation $d$ and $\Der(\cC)$ is contained in
$\{\phi\in\frso(\cC,\nup)\mid \phi(1)=0\}$. Hence the space of local derivations of
$\cC$ is contained in $\{\phi\in\frso(\cC,\nup)\mid \phi(1)=0\}$.

Conversely, given any $\Delta$ in $\{\phi\in\frso(\cC,\nup)\mid \phi(1)=0\}$
and any $x\in \cC$, if $x\in\FF 1$, then $\Delta(x)=0$, while if
$x\in\cC\setminus\FF 1$,
$\nup\bigl(x,\Delta(x)\bigr)=0=\nup\bigl(1,\Delta(x)\bigr)$, so
$\Delta(x)\in \bigl(\FF x\bigr)^\perp\cap\cC_0$.
Hence $\Delta(x)\in \Der(\cC)x$ by Lemma \ref{le:derC}, and it follows that there is a derivation $d\in\Der(\cC)$ such that $\Delta(x)=d(x)$.
\end{proof}

Note that Theorem~\ref{th:main} implies the following result (see \cite[Theorem 1.1]{AKA21}).

\begin{corollary}
The space $\LocDer\left(\OO\right)$ of all local derivations on the real Cayley algebra $\OO$ equipped with the Lie bracket is isomorphic to the Lie algebra
$\frso_7(\RR)$ of all real skew-symmetric matrices of order $7$.
\end{corollary}

\begin{remark}
{\rojo Note that  the dimensions  of the Lie algebras  $\LocDer\left(\cC\right)\cong \frso(\cC_0, \nup)$ and
$\Der\left(\cC\right)$  are equal to $21$ and $14$, respectively.} Therefore the  Cayley algebra $\cC$ admits local derivations which are not derivations.
\end{remark}

\

\section{2-Local derivations on Cayley algebras}

As we mentioned in Section 2, each Cayley algebra is either {\rojo the} split algebra (if the norm is isotropic) or a division algebra (if the norm is anisotropic) (see \cite{EKmon}, \cite{SV2000}).
We shall consider $2$-local derivations separately for each case.

\subsection{2-Local derivations on split  Cayley algebras}

\

\begin{theorem}\label{2localsplit}
	Any $2$-local derivation of the split Cayley algebra $\cC$ over a field
	$\FF$ is a derivation.
\end{theorem}

The proof of the following lemma  is similar to the proof of \cite[Lemma 3.4]{AKA21}, we include it for the sake of completeness.

\begin{lemma}\label{2-locallinear}
Any $2$-local derivation on an arbitrary Cayley algebra is linear.	
\end{lemma}

\begin{proof} Note that any $2$-local derivation $\Delta$ on $\cC$ is {\rojo `homogeneous': $\Delta(\lambda x)=\lambda\Delta(x)$ for all
$\lambda\in\FF$ and $x\in\cC$.
Indeed,  take a derivation} $d_{x, \lambda x}$ on $\cC$ such that $\Delta(x)=d_{x, \lambda x}(x)$ and $\Delta(\lambda x)=d_{x, \lambda x}(\lambda x)$. Then
\begin{align*}
\Delta(\lambda x) & = d_{x, \lambda x}(\lambda x)=
\lambda d_{x, \lambda x}(x)=\lambda \Delta(x).
\end{align*}
	
Let us show that $\Delta$ is additive.

Let $x,z\in \cC$. Take a derivation $d_{x, z}$  such that
$\Delta(x)=d_{x, z}(x)$ and $\Delta(z)=d_{x, z}(z)$.
From Equation \eqref{der-skew} we obtain
\begin{align*}
\nup(\Delta(x), z) = \nup(d_{x,z}(x), z) = -\nup(x, d_{x,z}(z)) =
-\nup(x, \Delta(z)).
\end{align*}
So,
\begin{align}\label{2-loc-skew}
\nup(\Delta(x), z) = -  \nup(x, \Delta(z)).
\end{align}
Now replacing the  element  $x$ by $x+y$ in \eqref{2-loc-skew}, we {\rojo obtain} that
	\begin{align*}
		\nup(\Delta(x+y), z) & = -\nup(x+y, \Delta(z))=
		-\nup(x, \Delta(z))-\nup(y, \Delta(z))\\
		& = \nup(\Delta(x), z)+\nup(\Delta(y),z)=
		\nup(\Delta(x)+\Delta(y), z).
	\end{align*}
	Thus
	\begin{align*}
		\nup(\Delta(x+y)-\Delta(x)-\Delta(y), z) & = 0
	\end{align*}
	for all $z\in \cC$. Since
	the bilinear  form $\nup$ is nondegenerate, it follows that
	$\Delta(x+y)=\Delta(x)+\Delta(y)$.
\end{proof}

So, any $2$-local derivation on a Cayley algebra  is linear and annihilates $1$.

\begin{corollary}\label{2local_linear}
Any $2$-local derivation on an arbitrary  Cayley algebra $\cC$ is a local derivation, and hence belongs to the orthogonal Lie algebra $\frso(\cC_0,\nup)$.
\end{corollary}

\begin{lemma}\label{le:eigenvectors}
Let $\Delta$ be a $2$-local derivation on $\cC$. If $a,b\in\cC$ and $\lambda \in\FF$ satisfy $\Delta(a)=0$, $ab=\lambda b$ (respectively $ba=\lambda b$), then $a\Delta(b)=\lambda \Delta(b)$ (respectively $\Delta(b)a=\lambda\Delta(b)$.
\end{lemma}
\begin{proof}
Let $d_{a,b}$ be a derivation on $\cC$ which coincides with $\Delta$ on $a$ and $b$. Then,
if $ab=\lambda b$ (the other case is similar) we get
{\rojo
\begin{multline*}
\lambda\Delta(b)  =\lambda d_{a,b}(b)=d_{a,b}(\lambda b)=
d_{a,b}(ab)\\
=  d_{a,b}(a)b+ad_{a,b}(b)=\Delta(a)b+a\Delta(b)=a\Delta(b),
\end{multline*}}
as desired.
\end{proof}

\begin{proof}[Proof of Theorem \ref{2localsplit}.]
Let $\cC$ be the split Cayley algebra over the field $\FF$. Consider the $\ZZ_3$-grading $\cC=(\FF e_1+\FF e_2)\bigoplus \cU\bigoplus \cV$, with the notations of Section 2. If $\Delta$ is a $2$-local derivation, let $d\in\Der(\cC)$ such that
$d(e_1)=\Delta(e_1)$. Substituting $\Delta$ by $\Delta-d$ we may assume
$\Delta(e_1)=0$, and hence $\Delta(e_2)=\Delta(1-e_1)=0$ too.

Since $\cU=\{x\in\cC\mid e_1x=x,\ xe_1=0\}$, Lemma \ref{le:eigenvectors} shows that $\Delta(\cU)\subseteq \cU$. In the same vein we get $\Delta(\cV)\subseteq \cV$.

Hence $\Delta\in\{f\in\frso(\cU\oplus \cV,\nup)\mid f(\cU)\subseteq \cU,\ f(\cV)\subseteq \cV\}$, where $\frso(\cU\oplus \cV,\nup)$ is identified with the subspace of the endomorphisms
in $\frso(\cC,\nup)$ that annihilate $\FF e_1\oplus\FF e_2$.

Since $\cU$ and $\cV$ are dual relative to $\nup$, the elements of
$\{f\in\frso(\cU\oplus \cV,\nup)\mid f(\cU)\subseteq \cU,\ f(\cV)\subseteq \cV\}$ are determined by its restriction to $\cU$. Actually, if $\{u_1,u_2,u_3\}$ and $\{v_1,v_2,v_3\}$ are dual bases in $\cU$ and $\cV$, then in these bases, the matrix of the restriction of $f$ to $\cV$ is
minus the transpose of the matrix of the restriction of $f$ to $\cU$. Note that
$\{e_1,e_2,u_1,u_2,u_3,v_1,v_2,v_3\}$ is a canonical basis of $\cC$.

However, relative to the $\ZZ_3$-grading above, $\Der(\cC)\subo$ is the subspace of the derivations that preserve $\cU$ and $\cV$ and annihilate $\FF e_1\oplus \FF e_2$, and this consists of the elements above with the extra condition of the trace of the restriction of $f$ to $\cU$ being $0$. Thus we have
\[
\{f\in\frso(\cU\oplus \cV,\nup)\mid f(\cU)\subseteq \cU,\ f(\cV)\subseteq \cV\}
=\Der(\cC)\subo\oplus \FF\varphi,
\]
where $\varphi(u_3)=u_3$, $\varphi(v_3)=-v_3$, and $\varphi(e_i)=\varphi(u_i)=\varphi(v_i)=0$ for $i=1,2$.

Therefore, it is enough to prove that $\varphi$ is not a $2$-local derivation.
For this, take $a=u_1-v_1$, so that $\varphi(a)=0$, and $b=u_2+v_3$. Then
$ab=(u_1-v_1)(u_2+v_3)=v_3+u_2=b$. Now, according to Lemma \ref{le:eigenvectors},
if $\varphi$ {\rojo were} a $2$-local derivation, we would get
$a\varphi(b)=\varphi(b)$. But $\varphi(b)=\varphi(u_2+v_3)=-v_3$, so that
$a\varphi(b)=(u_1-v_1)(-v_3)=-u_2\neq \varphi(b)$, a contradiction.
\end{proof}

Since any Cayley algebra over an algebraically closed field is split, Theorem~\ref{2localsplit} implies the following result (see \cite[Theorem 1.2]{AKA21}).

\begin{corollary}
Any $2$-local derivation of a Cayley algebra $\cC$ over an algebraically closed field is a derivation.
\end{corollary}

\smallskip

\subsection{2-Local derivations on division Cayley algebras}

\

The situation for Cayley division algebras is completely different.
In order to deal with it, some preliminaries are needed.

Let $\cC$ be the split Cayley algebra over an arbitrary field, and take a canonical basis
$\{e_1,e_2,u_1,u_2,u_3,v_1,v_2,v_3\}$ as in Table~\ref{table:1}.

\begin{lemma}\label{le:der_split}
	Let $\cC$ be the split Cayley algebra over an arbitrary field $\FF$. Consider the
	following subalgebras of $\Der(\cC)$:
	\[
	\cL=\{d\in\Der(\cC)\mid d(e_1)=d(e_2)=0\},\quad
	\cL'=\{d\in\Der(\cC)\mid d(u_1+v_1)=0\}.
	\]
	Then the dimension of each of the following subspaces:
$\cL(u_1+v_1)$,
	$\cL'(u_1+\mu v_1)$ with $1\neq \mu\in\FF$, and $\cL'(u_2+v_2)$,  is always $5$.
\end{lemma}
\begin{proof}
	These subspaces are easily computed using Equation \eqref{eq:derivation_matrix}, obtaining:
	\begin{itemize}
		\item $\cL(u_1+v_1)=\espan\langle u_1-v_1,u_2,u_3,v_2,v_3\rangle$,
		\item $\cL'(u_1+\mu v_1)=\espan\langle e_1-e_2,u_2,u_3,v_2,v_3\rangle$ ($\mu\neq 1$), and
		\item $\cL'(u_2+v_2)=\espan\langle e_1-e_2,u_1+v_1,u_2-v_2,u_3,v_3\rangle$.
	\end{itemize}
	The result follows.
\end{proof}

\begin{corollary}\label{co:division}
	Let $\cC$ be a Cayley division algebra over an arbitrary field. Let $x,y\in\cC$ be elements such that $1$, $x$, and $y$, are linearly independent. Consider the following subalgebra of $\Der(\cC)$: $\cS=\{d\in\Der(\cC)\mid d(x)=0\}$.
	Assume that either:
	\begin{itemize}
		\item The restriction of the norm $\nup$ to $\cK=\FF 1+\FF x$ is nonsingular  (that is, $\cK$ is a two-dimensional composition subalgebra, and note that this is always the case if the characteristic of $\FF$ is not $2$), or
		\item The characteristic of $\FF$ is $2$ and $\nup(1,x)=0=\nup(1,y)$.
	\end{itemize}
	Then the subspace $\cS y$ equals $(\FF 1+\FF x+\FF y)^\perp$.
\end{corollary}
\begin{proof}
	Note first that since $\Der(\cC)$ is contained in the orthogonal Lie algebra
	$\frso(\cC,\nup)$, we always have $\cS y\subseteq (\FF 1+\FF x+\FF y)^\perp$, and this is a subspace of dimension $5$ because $1$, $x$ and $y$ are linearly independent. Therefore, it is enough to check that the dimension of $\cS y$ is $5$.
	Also, in order to check this, we may extend scalars to an algebraic closure
	$\overline{\FF}$.
	
	Assume first that $\cK=\FF 1+\FF x$ is a composition subalgebra of $\cC$, then $y=a+y'$ with $a\in\cK$ and $0\neq y'\in \cK^\perp$. Since $\cC$ is a division algebra,
	$\nup$ is anisotropic, so $\nup(y')\neq 0$. Note that $\cS y=\cS y'$, because $\cS$ annihilates both $1$ and $x$. Hence we may assume $y\in\cK^\perp$. Extending scalars to $\overline{\FF}$, $\cK\otimes_\FF\overline{\FF}$ is a two-dimensional composition
	subalgebra, and hence it has two orthogonal idempotents:
	$\overline{\cK}\bydef\cK\otimes_\FF\overline{\FF}=\overline{\FF}e_1\oplus\overline{\FF}e_2$. These two idempotents can be completed to a canonical basis
	$\{e_1,e_2,u_1,u_2,u_3,v_1,v_2,v_3\}$ of
	$\overline{\cC}\bydef\cC\otimes_\FF\overline{\FF}$. Denote again by $\nup$ the norm in $\overline{\cC}$.
	If $\nup(y)=\mu$ and we pick a square root $\alpha\in\overline{\FF}$ of $\mu$, then
	$\nup(y)=\nup\bigl(\alpha(u_1+v_1)\bigr)$ and, using the Cayley-Dickson doubling
	process we can easily define an automorphism $\varphi$ of $\overline{\cC}$ that
	is the identity on $\overline{\cK}$ and satisfies
	$\varphi(y\otimes 1)=\alpha(u_1+v_1)$. Then
	$\varphi(\cS\otimes_\FF\overline{\FF})\varphi^{-1}
	=\{d\in\Der(\overline{\cC})\mid d(e_1)=0=d(e_2)\}$, and the result follows from
	Lemma \ref{le:der_split}.
	
	Assume now that the characteristic is $2$ and $\nup(1,x)=0=\nup(1,y)$. In particular $\bar x=-x=x$, and we have two different possibilities depending on whether
	$\nup(x,y)$ is $0$ or not.
	
	If $\nup(x,y)\neq 0$, then $\nup(1,yx)=\nup(\bar y x,1)=\nup(y,x)\neq 0$, and hence $\cK=\FF 1+\FF yx$ is a composition subalgebra of $\cC$. Moreover,
	$\nup(yx,x)=\nup(y,x^2)=0$, as $x^2=-\nup(x)1$, and $(yx)x=\nup(x)y$, so
	$y\in \cK x$. Extending scalars as in the previous case, it turns out that there is
	an automorphism $\varphi$ of $\overline{\cC}$ that is the identity on
	$\overline{\cK}$ and such that $\varphi(x\otimes 1)=\alpha(u_1+v_1)$ for some
	$0\neq \alpha\in\overline{\FF}$. Since $y$ is in $\cK x$, we get
	$\varphi(y\otimes 1)\in(\overline{\FF}e_1+\overline{\FF}e_2)(u_1+v_1)=
	\overline{\FF}u_1+ \overline{\FF}v_1$. Then we have $\varphi(y\otimes 1)\in \overline{\FF}(u_1+\mu v_1)$, with $\mu$ different from $0$, as $\nup(y)\neq 0$ because $\nup$ is anisotropic, and different from $1$, as $x$ and $y$ are linearly independent. In this case
	$\varphi(\cS\otimes_\FF\overline{\FF})\varphi^{-1}
	=\{d\in\Der(\overline{\cC})\mid d(u_1+v_1)=0\}$, and the result follows from
	Lemma \ref{le:der_split}.
	
	Finally, if $\nup(x,y)=0$, pick an element
	$u\in (\FF x+\FF y)^\perp\setminus (\FF 1)^\perp$ (remember that $1$, $x$, and $y$
	are linearly independent). Then $\nup(1,u)\neq 0$, so $\cK=\FF 1+\FF u$ is
	a composition subalgebra of $\cC$. Here $\nup(x,\cK)=0=\nup(y,\cK)$. Consider the quaternion subalgebra $\cQ=\cK\oplus\cK x$.
	As $y$ is orthogonal to $\cK$ and to $x$, it belongs to $(\cK+\FF x)^\perp=\FF x+\cQ^\perp$, hence there exists a scalar $\alpha\in\FF$ such that
	$y-\alpha x\in\cQ^\perp$. But $\cS y=\cS(y-\alpha x)$, so we may assume
	$y\in\cQ^\perp$.
	Extending scalars and using the Cayley-Dickson doubling process,
 we can easily define an automorphism $\varphi$ of $\overline{\cC}$ such that
	$\varphi(x)\in\overline{\FF}(u_1+v_1)$ and $\varphi(y)\in\overline{\FF}(u_2+v_2)$.
	Again we get $\varphi(\cS\otimes_\FF\overline{\FF})\varphi^{-1}
	=\{d\in\Der(\overline{\cC})\mid d(u_1+v_1)=0\}$, and the result follows from
	Lemma \ref{le:der_split}.
\end{proof}

\begin{remark}
	If $\cC$ is a Cayley algebra and $\cK$ is a two-dimensional composition subalgebra, the subalgebra $\{d\in\Der(\cC)\mid d(\cK)=0\}$ is a special unitary Lie algebra.
	(See \cite[\S 5]{BDE} and references therein.)
\end{remark}

Now we are ready to tackle the case of $2$-local derivations on  Cayley division
algebras.

The following {\rojo result, combined with Corollary~\ref{2local_linear}, shows that for Cayley division algebras the notions of local} and
$2$-local derivations coincide.

\begin{theorem}\label{th:2local_division}
	Let $\cC$ be a Cayley division algebra over an arbitrary field $\FF$. Then any
	local derivation of $\cC$ is a $2$-local derivation. As a consequence, the space of $2$-local
	derivations coincide with
	$\{d\in\frso(\cC,\nup)\mid d(1)=0\}\simeq\frso(\cC_0,\nup)$.
\end{theorem}

\begin{proof}
	Take an element $\Delta\in\frso(\cC,\nup)$ with $\Delta(1)=0$, and take two elements
	$x,y\in\cC$. We must show that there is a derivation $d\in\Der(\cC)$ such that
	$\Delta(x)=d(x)$ and $\Delta(y)=d(y)$. In other words, $\Delta$ coincides with $d$
	on $\FF 1+\FF x+\FF y$. If $x\in\FF 1$ this is trivial, as $\Delta(1)=d(1)=0$ and $\Delta$ is a local derivation. If $y\in\FF 1+\FF x$ this is trivial too. Hence assume $1$, $x$, and $y$ are linearly independent.
	
	Since $\Delta$ is a local derivation, there is a derivation $d'\in\Der(\cC)$
	such that $\Delta(x)=d'(x)$, and hence $\Delta'=\Delta-d'$ annihilates $1$ and $x$.
	As $\Delta'$ is in $\frso(\cC,\nup)$, $\Delta'(y)$ lies in $(\FF 1+\FF x+\FF y)^\perp$,
	so Corollary \ref{co:division} shows the existence of a derivation $d\in\Der(\cC)$
	with $d(x)=0$ such that $\Delta'(y)=d(y)$. Hence $\Delta$ and the derivation
	$d'+d$ coincide on $x$ and $y$, as desired.
\end{proof}

\begin{remark}
Since $\twoLocDer\left(\cC\right)=\LocDer\left(\cC\right)\cong \frso(\cC_0, \nup)$,  it follows that any  division Cayley algebra $\cC$ admits $2$-local derivations which are not derivations.
\end{remark}

\

\section{Local and $2$-local derivations on $7$-dimensional simple Malcev algebras}

An algebra $\cA$ over a field $\FF$ is called a Malcev algebra if its multiplication is anticommutative and satisfies the following identity:
$$
{\displaystyle (xy)(xz)=((xy)z)x+((yz)x)x+((zx)x)y.}
$$
Let $\cC$ be a Cayley algebra over a field $\FF$ of {\rojo characteristic
$\neq 2$ and define a bracket $[\cdot, \cdot]$ by
$$
[x, y]=\frac{1}{2}(xy-yx),\, x, y\in \cC.
$$
Then the space $\cC_0^{-}=\left\{x\in \cC: t(x)=0\right\}$  with the bracket $[\cdot, \cdot]$,  is a simple Malcev algebra (see, e.g., \cite[Theorem 4.10]{EM94}).}

{\rojo If the characteristic of the ground field $\FF$ is $\neq 2,3$,
$\cC_0^-$ is a
central simple non-Lie Malcev algebra and, conversely, any such algebra is obtained, up to
isomorphism, in this way. (See \cite{Filippov_1}.)

However, if the characteristic of $\FF$ is $3$, then any simple Malcev algebra over
$\FF$ is a Lie algebra (\cite{Filippov_2}). In this case, $\cC_0^-$
is a central simple Lie algebra of type $A_2$ (i.e., a twisted form of
$\frpsl_3(\FF)$) and, conversely, any central simple Lie algebra of type $A_2$ is
obtained in this way. (See \cite{Alberca_et_al} or \cite[Theorem 4.26]{EKmon}.)}

It is known \cite[Chapter IV, Lemma 6.1]{EM94} that  every derivation   on a Cayley
algebra $\cC$ defines a derivation on $\cC_0^{-}$, and conversely.
More precisely, if  $d$ is a derivation on $\cC$, then the restriction
$d\vert_{\cC_0^{-}}$ is  a derivation on $\cC_0^{-}$. Conversely, if $d$ is a derivation on $\cC_0^{-}$,
then its extension to  $\cC$ {\rojo defined}
as
$$
d(\lambda 1+x)=d(x), \, \lambda\in \FF, \, x\in \cC_0^{-},
$$
is a derivation on $\cC$.
Therefore a similar  correspondence between local ($2$-local) derivations
of the Cayley algebra $\cC$ and the Malcev algebra $\cC_0^{-}$ are also true.

So, from Theorems~\ref{th:main}, \ref{2localsplit} and \ref{th:2local_division}  we obtain the following results.

\begin{theorem}
Let $\cC_0^{-}$ be {\rojo the central simple Malcev algebra associated with a Cayley algebra $\cC$ over a field $\FF$ of characteristic $\neq 2$.}  Then
\begin{enumerate}
\item the space of all local derivations of $\cC_0^{-}$ is the Lie algebra
$\frso(\cC_0,\nup);$
\item if $\cC$ is a split algebra, then every $2$-local derivation on $\cC_0^{-}$ is a derivation;
\item if $\cC$ is a division algebra, then
$$
\twoLocDer(\cC_0^{-})\cong \LocDer(\cC_0^{-})\cong
\frso(\cC_0,\nup).
$$
\end{enumerate}
\end{theorem}

\begin{remark}
A Cayley  algebra 	$\cC$ over a field $\FF$ of {\rojo characteristic
$\neq 2$} equipped with the multiplication
$$
x\circ y=\frac{1}{2}(xy+yx)
$$
becomes a Jordan algebra $\left(\cC^+, \circ\right)$.	Using the description
of derivations  of the algebra $\cC^+$ \cite[Chapter IV, Page 175]{EM94}, we obtain the following isomorphisms:
$$
\LocDer(\cC)\cong \LocDer(\cC_0^{-})\cong \Der(\cC^+) \cong \frso(\cC_0,\nup).
$$
\end{remark}

\end{document}